\documentclass[a4paper,12pt]{amsart}

\usepackage{epsfig}
\usepackage{amsmath,amsthm}
\usepackage{amsfonts}
\usepackage{latexsym}
\usepackage{amsbsy}
\usepackage{amssymb,mathscinet}
\usepackage{amsfonts,amssymb,amscd,amsmath,enumerate,verbatim,calc,eucal,enumerate}

\numberwithin{equation}{section}

\usepackage{amsfonts,amssymb,amscd,amsmath,enumerate,verbatim,calc,eucal,enumerate}
\newcommand{\ma}{\mathcal}
\newcommand{\mf}{\mathfrak}
\newcommand{\m}{\CMcal}

\theoremstyle{plain}
\newtheorem{theorem}{Theorem}[section]

\newtheorem{definition}[theorem]{Definition}

\newtheorem{remark}[theorem]{Remark}

 \title{KSGNS construction for $\tau$-maps on S-modules and $\mf{K}$-families}

     \author{Santanu Dey}
     \address{Department of Mathematics, Indian Institute of Technology Bombay, Mumbai-400076, India}
     \email{santanudey@iitb.ac.in}
     \thanks{}

     \author{Harsh Trivedi}
     \address{Department of Mathematics, Indian Institute of Technology Bombay, Mumbai-400076, India}
     \email{harsh1@iitb.ac.in}
     \thanks{}

     \date{\today}
\begin{document}

     \begin{abstract} 
We introduce S-modules, generalizing the notion of Krein $C^*$-modules, where a fixed unitary replaces
the symmetry of Krein $C^*$-modules. The representation theory 
on S-modules is explored and for a given $*$-automorphism $\alpha$ on a $C^*$-algebra the KSGNS construction for $\alpha$-completely 
positive maps is proved. An extention of this theorem for $\tau$-maps
is also achieved, when $\tau$ is
an $\alpha$-completely positive map, along with a decomposition theorem for $\mf K$-families.

     \end{abstract}
\maketitle

\section{Introduction}

 A {\it symmetry} on a Hilbert space is a bounded operator $J$ such that $J=J^*=J^{-1}$. A Hilbert space along with a symmetry, 
forms a {\it Krein space} where the symmetry induces an indefinite inner-product on the space. Dirac \cite{D42} and Pauli \cite{P43} were among the 
pioneers to 
explore the quantum field theory using Krein spaces.

In quantum field theory one
encounters Wightman functionals which are
positive linear functionals on Borchers algebras (cf. \cite{B62}).
In the massless or the guage quantum field theory, Strocchi showed that the locality and  positivity both cannot be assumed together in a model. The axiomatic field theory motivates theoretical physicists to keep the locality assumption and sacrifice positivity by  
considering  indefinite inner products (cf. \cite{B77}), and more specifically Krein spaces (cf. \cite{B74}), in the gauge field theory. In this context Jakobczyk
defined the $\alpha$-positivity, where $\alpha$ is a  $*$-automorphism of Borchers algebra, in  \cite{Ja84} and derived a reconstruction theorem for 
Strocchi-Wightman states.

\begin{definition}
 Let $\m B$ be a $*$-subalgebra of a unital $*$-algebra $\m A$ containing the unit. Assume $P:\m A\to\m B$ to be a conditional expectation 
 (i.e. $P$ is a linear map preserving the 
 unit and the involution such that $P(bab')=bP(a)b'$ for each $a\in\m A$; $b\in\m B$). A Hermitian linear functional 
 $\tau$ defined on $\m A$ is called a {\rm $P$-functional} if the following holds:
 \begin{itemize}
  \item [(i)] $\tau\circ P=\tau$,
  \item [(ii)] $2\tau (P(a)^*P(a))\geq \tau(a^*a)$ for all $a\in\m A$.
 \end{itemize}
 \end{definition}
If we define a linear mapping $\alpha:\m A\to\m A$ by $\alpha(a)=2P(a)-a$ for each $a\in\m A$, then the $P$-functional $\tau$ satisfies
\begin{align*}
\tau(\alpha(a)\alpha(a'))= \tau(aa')~\mbox{and}~\tau(\alpha(a)^*a)\geq 0~\mbox{for all}~a,a'\in\m A.
\end{align*}
Thus $P$-functionals generalize $\alpha$-positivity. The Gelfand-Naimark-Segal (GNS) construction 
for states, is a 
fundamental result in operator theory, which illustrates how using a state on a $C^*$-algebra we can obtain a cyclic representation
of that $C^*$-algebra 
on a Hilbert space. Antoine and Ota \cite{AO89} proved that using $P$-functionals one can obtain unbounded GNS
representations of a $*$-algebra on
a Krein space.

The usefulness of $\alpha$-positivity in gauge field theory led  to the concept of the $\alpha$-completely positive maps which were introduced in \cite{HJU10} and for them a KSGNS type construction 
was carried out on certain modules called Krein $C^*$-modules where $\alpha^2 = id$ (i.e., order of $\alpha$ is 2). We extrend this study of $\alpha$-completely positive maps for $\alpha$ not necessarily of order 2  and obtain a KSGNS type construction on a bigger class of modules called $S$-modules.
To illustrate KSGNS construction we first need to recall some notions: We say that a linear map $\tau$ from a $C^*$-algebra $\m B$ to a $C^*$-algebra $\m C$ is {\it completely positive} if
$$\sum_{i,j=1}^n c_j^{*}
\tau(b_j^{*}b_i)c_i\geq 0$$ for all $b_1,b_2,\ldots,b_n\in\m B$; $c_1,c_2,\ldots,c_n\in\m C$ and $n\in \mathbb N$. The completely positive
maps are crucial to the study of the classification of $C^*$-algebras, the classification of $E_0$-semigroups, etc. The Stinespring theorem 
(cf. \cite{St55}) characterizes completely positive maps and if we choose the completely positive maps to be states it reduces to the GNS construction.

A {\it Hilbert $C^*$-module over a $C^*$-algebra $\m B$} or a  {\it Hilbert $\m B$-module} is a complex vector space which is a
right $\m B$-module such that the module action is compatible with the scalar product and there is a mapping $\langle
\cdot,\cdot \rangle : E \times E \to \m{B}$ satisfying the following conditions:
\begin{itemize}
 \item [(i)] $\langle x,x \rangle \geq 0 ~\mbox{for}~ x \in {E}  $ and $\langle x,x \rangle = 0$ only if $x = 0 ,$
\item [(ii)] $\langle x,yb \rangle = \langle x,y \rangle b ~\mbox{for}~ x,y \in {E}$ and $~\mbox{for}~ b\in \m B,  $
 \item [(iii)]$\langle x,y \rangle=\langle y,x \rangle ^*~\mbox{for}~ x ,y\in {E} ,$
\item [(iv)]$\langle x,\mu y+\nu z \rangle = \mu \langle x,y \rangle +\nu \langle x,z \rangle ~\mbox{for}~ x,y,z \in {E} $ and for 
$\mu,\nu \in \mathbb{C}$
\item [(v)] $E$ is complete with respect to the norm $\| x\| :=\|\langle
x,x\rangle\|^{1/2} ~\mbox{for}~ x \in {E}$. 
\end{itemize}

\begin{definition}
 Let $E$ and $ F$ be Hilbert $\m
A$-modules over a $C^*$-algebra $\m A$. For a given map $S:E\to F$
if there exists a map $S': F\to E$ such
that
 \[
 \langle S (x),y\rangle =\langle x,S'(y) \rangle~\mbox{for all}~x\in E, y\in  F,
 \]
then $S'$ is unique for $S$, and we say $S$ is {\rm adjointable} and denote $S'$ by $S^{*}$. 
Every adjointable map $S:E\to F$ is {\rm right $\m A$-linear}, i.e., $S(xa)=S(x)a$ for all $x\in E,$ $a\in\m A$. The symbol
$\ma B^a (E,F)$ denotes the collection of all 
adjointable maps from $E$ to $ F$. We use $\ma B^a (E)$ for $\ma B^a (E,E)$. The {\rm strict topology} on $\ma B^a ({E})$ is the topology induced
by the seminorms $a\mapsto \|ax\| $, $a\mapsto \|a^*y\| $ for each
$x,y\in E$. 
\end{definition}
Kasparov \cite{Kas80} obtained the following theorem, called {\it Kasparov-Stinespring-Gelfand-Naimark-Segal (KSGNS) 
construction} (cf. \cite{La95}), which is a dilation theorem for strictly continuous completely positive maps: 

\begin{theorem}
 Let $\m B$ and $\m C$ be $C^*$-algebras. Assume $E$ to be a Hilbert $\m C$-module and $\tau:\m B\to \ma B^a ({E})$ to be a strictly continuous 
 completely positive map. Then we get a Hilbert $\m C$-module $F$ with a $*$-homomorphism $\pi:\m B\to \ma B^a (F)$ and $V\in\ma B^a (E,F)$ such that
 $\overline{span}~{\pi(\m B)VE}=F$ and 
 \[
  \tau(b)=V^*\pi(b)V~\mbox{for all}~b\in\m B.
 \] 
\end{theorem}

Let $(E,\langle\cdot,\cdot\rangle)$ be a Hilbert $C^*$-module over a $C^{\ast}$-algebra $\m{A}$ and let $J$ be a fundamental
symmetry on $E$, i.e., $J$ is an invertible adjointable map on $E$ such that $J=J^*=J^{-1}$. Define an
$\m{A}$-valued indefinite inner product on $E$ by
\begin{eqnarray}
[x,y]: =\langle Jx,y\rangle~\mbox{for all}~x,y\in E.
\end{eqnarray}
In this case we say $(E,\m A,J)$ is a {\it Krein $\m{A}$-module}.
If $\m{A}=\mathbb{C}$, then $(E,\mathbb{C},J)$ is a {\rm Krein
space} and in addition if $J$ is the identity operator, then it becomes a Hilbert space. In the definition of Krein spaces if we 
replace the symmetry $J$ by a unitary,
then we get {\it $S$-spaces}. Szafraniec introduced the notion of S-spaces in \cite{Sz09}. Phillipp, Szafraniec and Trunk \cite{PST11} 
investigated invariant subspaces of selfadjoint operators
in Krein spaces by using results obtained through a detailed analysis of S-spaces.  We introduce the notion of S-modules below:

\begin{definition}
Let $(E,\langle\cdot,\cdot\rangle)$ be a Hilbert $C^*$-module over a $C^{\ast}$-algebra $\m{A}$ and let $U$ be a unitary on $E$,
i.e., $U$ is an invertible adjointable map on $E$ such that $U^*=U^{-1}$. Then we can define an
$\m{A}$-valued sesquilinear form by
\begin{eqnarray}
[x,y]: =\langle x,Uy\rangle~\mbox{for all}~x,y\in E.
\end{eqnarray}
In this case we say $(E,\m A, U)$ is an {\rm S-module}.
\end{definition}

If $U=I$, then $\langle \cdot,\cdot\rangle$ and $[\cdot,\cdot]$ coincide for the S-module $(E,\m A, U)$. In the case when $U=U^*$, the S-module 
$(E,\m A, U)$ forms a
Krein $\m A$-module. The following is the definition
of an $\alpha$-completely positive map which will play an
important role in this article:

\begin{definition}\label{def21}
Let $\m A$ be a unital $C^*$-algebra and $\alpha:\m A\to\m A$ be an automorphism, i.e., 
$\alpha$ is a unital bijective $*$-homomorphism.
Let $\m B$ be a $C^*$-algebra and $E$ be a Hilbert $\m B$-module. If $\left( E
,\m B,U\right) $ is an S-module, then a map $\tau :\m{A}\rightarrow \mathcal{B}^a(%
{E})$ is called {\rm $\alpha $-completely positive (or $\alpha
$-CP)} if it is a $*$-preserving map such that

\begin{enumerate}
\item[(i)]$\tau \left( \alpha (a)\right)=U^*\tau (a)U=\tau (a)$
for all $a\in \m{A}$;

\item[(ii)] $\sum\limits_{i,j=1}^{n}\left\langle {x}
_{i},\tau \left( \alpha \left( a ^{\ast }_{i}\right)a_{j}\right) {x}
_{j}\right\rangle \geq 0$ for all $n\geq 1$; $a_{1},\ldots ,a_{n}\in %
\m{A}$ and ${x} _{1},\ldots ,{x} _{n}\in E$;

\item[(iii)] for any $a\in \m{A}$, there is $M(a)>0$ such that
\begin{equation*}
\sum\limits_{i,j=1}^{n}\left\langle {x}_i,\tau \left(\alpha \left( a^*_{i}a^*\right)aa_{j}\right){x}_j \right\rangle
\leq M(a)\sum\limits_{i,j=1}^{n}\left<{x_i}, \tau \left( \alpha(a^{*}_{i})a_{j}\right){x}_j \right>
\end{equation*}%
for all $n\geq 1$; ${x}_{1},\ldots ,{x}_{n}\in E$ and $a_{1},\ldots ,a_{n}\in \m{A}$.
\end{enumerate}

\end{definition}

Several operator theorists and mathematical physicists recently explored the dilation theory of $\mathcal{B}^a(%
{E})$-valued $\alpha$-CP maps using Krein $C^{\ast}$-modules
(cf. \cite{HJ10, HJU10}) where $E$ is a Hilbert $C^*$-module and $\alpha^2=id_{\m A}$. 
In this article we explore the KSGNS construction of certain maps between 
Hilbert $C^*$-modules which we define below:

\begin{definition}
 Let ${E}$ be a Hilbert $\m A$-module, ${F}$ be a Hilbert $\m
B$-module and $\tau$ be a linear map from $\m A$ to $\m B$. A map ${T}:E\to F$ is called {\rm $\tau$-map} if
$$\langle T(x),T(y)\rangle=\tau(\langle x,y\rangle)~\mbox{for all} ~x,y\in{E}.$$  
\end{definition}
\noindent The dilation theory of $\tau$-maps has been explored in \cite{BRS12}, \cite{Sk12}, \cite{SSu14}, \cite{HJK13}, \cite{UMM14}, etc.

For any Hilbert spaces $\m H$ and $\m K$, let $\ma B(\m H,\m K)$ denote the space of all bounded linear operators from $\m H$ to $\m K$.
Assume $E$ to be a Hilbert $\m
B$-module where $\m B$ is a von Neumann algebra acting on a Hilbert space $\m H$ and so $E\bigotimes \m H$ is a Hilbert space where 
$\bigotimes$ denotes the interior tensor
product. For a fixed $x\in E$ we define a bounded linear
operator $L_x:\m H\to E\bigotimes \m H$ by
$$L_x (h):=x\otimes h~\mbox{for}~ h\in \m H.$$
We have $L^*_{x_1} L_{x_2} =\langle x_1,x_2\rangle~\mbox{for all}~
x_1,x_2\in E$. This allows us to identify  each $x\in E$ with $L_x$ and thus
$E$ is identified with a concrete submodule of $\ma B(\m H,E\bigotimes \m H)$. We say that $E$ is a {\it von Neumann $\m B$-module}
or a {\it von Neumann
module over $\m B$} if $E$ is strongly closed in $\ma B(\m
H,E\bigotimes \m H)$. 
This notion of von Neumann modules is due to Skeide (cf. \cite{Sk06}). In fact, $a\mapsto a\otimes id_{\m H}$ is a representation of
$\ma B^a(E)$ on $E\bigotimes \m H$, and therefore it is an isometry. Thus we can consider $\ma B^a(E)\subset \ma B(E\bigotimes \m H)$
and in this way $\ma B^a(E)$ is a von Neumann algebra acting on $E\bigotimes \m H$. 
The von Neumann modules were used as a tool in \cite{BSu13} to explore 
Bures distance 
between two completely
positive maps. In \cite{H14},
we proved a Stinespring type theorem for
$\tau$-maps, when $\m B$ is any von Neumann algebra and $F$ is any von Neumann $\m B$-module. 
As in \cite{H14}, in this article too at certain places we work with von Neumann modules
instead of Hilbert $C^*$-modules because all von Neumann modules are self-dual, and hence they are complemented 
in all Hilbert $C^*$-modules which contain them as submodules. 

In \cite{UMM14}, for an order two automorphism $\alpha$ on a $C^*$-algebra $\m A$, the authors constructed
representations of $\m A$ on Krein $C^{\ast}$-modules by considering $\mathcal{B}^a(%
{E})$-valued $\alpha$-CP maps where $E$ is a Krein $C^*$-module and proved a KSGNS type construction for $\tau$-maps
where $\tau$ is an $\alpha$-CP map.
This is the motivation for the approach taken by us to study the representation theory of $\tau$-maps on $S$-modules in Section \ref{secK1}.

A $C^*$-algebra valued positive definite kernels were defined by Murphy in \cite{M97}. The reproducing kernel Hilbert $C^*$-modules,
which generalizes the terminology of reproducing kernel Hilbert spaces (cf. \cite{BCR84}), were analysed extensively in \cite{H08} by Heo.
Szafraniec \cite{S10} obtained a dilation theorem, 
which extends the 
Sz-Nagy's principal theorem \cite{RS90}, for certain $C^*$-algebra valued positive definite functions defined on $*$-semigroups. 
The KSGNS construction is a special case of Szafraniec's dilation theorem.

 Motivated by the 
 definition of $\tau$-map, we introduced the following notion of $\mf{K}$-family in \cite{DH14}:
   Let ${E}$ and ${F}$ be Hilbert $C^*$-modules over $C^*$-algebras $\m B$ and $\m C$ respectively. Assume $\Omega$ to be a set and  
   $ \mf{K}:\Omega\times \Omega\to \ma B(\m B,\m C)$ to be a kernel. 
Let $\ma K^{\sigma}$ be a map from $E$ to $F$ for each $\sigma\in \Omega$. The family $\{\ma K^{\sigma}\}_{\sigma\in \Omega}$ is called {\it $\mf{K}$-family} if
   \[
    \langle \ma K^{\sigma} (x),\ma K^{{\sigma}'}(x')\rangle=\mf{K}^{\sigma,\sigma'}(\langle x, x'\rangle)~\mbox{for}~x,x' \in E;~\sigma,\sigma'\in \Omega.
   \]
We obtain a partial factorization theorem for $\mf K$-families in Section 3, where $\mf K$ is an $\alpha$-CPD kernel, with the help of 
reproducing kernel S-correspondences. We recall the definition of CPD-kernel below:
\begin{definition}
Let $\m B$ and $\m C$ be $C^*$-algebras. By $\ma
B(\m B,\m C)$ we denote the set of all bounded linear maps from  $\m B$ to $\m C$. For a set $\Omega$
we say that a mapping $\mf{K}:\Omega\times \Omega \to \ma B(\m B,\m C)$ is a {\rm completely positive definite kernel} or a {\rm CPD-kernel} over $\Omega$ from
$\m B$ to $\m C$ if
\[
 \sum_{i,j} c^*_i \mf{K}^{\sigma_i, \sigma_j} (b^*_i b_j) c_j \geq 0~\mbox{for all finite choices of $\sigma_i\in \Omega$, $b_i\in \m B$, $c_i\in \m C$. }~
\]
\end{definition}
 Accardi and Kozyrev in \cite{AK01} considered semigroups of CPD-kernels over the set $\Omega = \{0, 1\}$. 
 Barreto, Bhat, Liebscher and Skeide \cite{BBLS04} studied several results regarding structure of type I 
 product-systems of Hilbert $C^*$-modules based on the dilation theory of CPD-kernels over any set $\Omega$. Their approach was based on the Kolmogorov decomposition of a CPD-kernel. Ball, Biswas, Fang 
 and ter Horst \cite{BBQH09} introduced the notion of reproducing kernel $C^*$-correspondences and identified Hardy spaces
 studied by Muhly-Solel \cite{MS04} with a reproducing kernel $C^*$-correspondence for a CPD-kernel which is an analogue 
 of the classical Szeg\"o kernel. Recently in \cite{ABR11}, module-valued coherent states were considered and it was proved 
 that these coherent states gives rise to a CPD-kernel which has a reproducing property.

\section{KSGNS type construction for $\tau$-maps}\label{secK1}

Assume $\left( {E}_{1},\m B,U_{1}\right) $ and $\left( E
_{2},\m B,U_{2}\right) $ to be S-modules, where $E_1$ and $E_2$ are Hilbert $C^*$-modules over a $C^*$-algebra $\m B$.
For each $T\in \mathcal{B}^a(%
{E}_{1},{E}_{2})$, there exists an operator $T^{\natural}\in \mathcal{B}^a(%
{E}_{2},{E}_{1})$ such that
\[
\langle T(x),U_2 y\rangle=\langle x,U_1T^\natural(y)\rangle~\mbox{for all}~x\in E_1,~y\in E_2.
\]
In fact, $T^{\natural}=U^*_{1}T^{\ast }U_{2}$. Suppose $\m{A}$ is a $C^{\ast }$-algebra and $(E,\m B,U)$ be an S-module.
An algebra homomorphism $\pi :\m{A}\rightarrow
\mathcal{B}^a({E})$ is called an {\it $U$-representation of $\m{A}$ on $(E,\m B,U)$} if
$\pi (a^{\ast })=U^*\pi (a)^{\ast }U=\pi (a)^{\natural}$, i.e., $$[\pi (a)x ,y]=[x ,\pi (a^{\ast })y]~\mbox{for all}~x,y\in E.$$


The theorems in this section are analogous to Theorem 3.2 of \cite{HJK13} and Theorem 4.4 of \cite{HJU10}, and Theorem 2.6 of \cite{UMM14}.

\begin{theorem}\label{THM1}
\label{main0} Let $\m A$ and $\m B$ be unital $C^*$-algebras and $\alpha:\m A\to\m A$ be an automorphism.
Suppose $(E_1,\m B,U_1)$ is an S-module where $E_1$ is a Hilbert $\m B$-module and $U_1$ is a unitary on $E_1$. If $\tau:%
\m{A}\rightarrow\mathcal{B}^a({E}_{1})$ is an $\alpha$-CP map, then there exist

\begin{enumerate}
\item [(i)] a Hilbert $\m B$-module $E_0$ and a unitary $U_0$ such that $\left( E%
_{0},\m B,U_{0}\right)$ is an S-module,
\item [(ii)] an
$U_0$-representation $\pi_0$ of $\m{A}$ on $\left( E%
_{0},\m B,U_{0}\right)$ satisfying $$V^*\pi_0(a)^*\pi_0(b)V =V^*\pi_0(\alpha(a)^*b)V~\mbox{for each}~a,b\in\m A,$$ a map $V\in \ma B^a (E_1,E_0)$ 
such that $V^{\natural}=V^{\ast },$ and
$$\tau(a)=V^{*}\pi_0(a)V~\mbox{for all}~a\in\m A.$$
\end{enumerate}
\end{theorem}

\begin{proof}
Let $\m{A}\bigotimes_{
{alg}}{E}_{1}$ be the algebraic tensor product of $\m{A}$ and ${E}_{1}$. Define a map $\langle\cdot,\cdot\rangle:(\m{A}\bigotimes _{
{alg}}{E}_{1})\times(\m{A}\bigotimes _{
{alg}}{E}_{1})\to \m B$ by
\[
\left\langle \sum\limits_{i=1}^{n}a_{i}\otimes {x}_{i},\sum\limits_{j=1}^{m}a'_{j}\otimes {y} _{j}\right\rangle
=\sum\limits_{i=1}^{n}\sum\limits_{j=1}^{m}\left\langle {x} _{i},\tau
\left( \alpha \left( a^{\ast }_{i}\right) a'_{j}\right) {y}
_{j}\right\rangle 
\]
for all $a_1,\ldots,a_n;~a'_1,\ldots,a'_m\in \m A$ and $ {x}_1,\ldots, {x}_n;~{y}_1,\ldots,{y}_m\in E_1$. The
condition (ii) of Definition \ref{def21} implies that $\langle~,~\rangle$ is a positive definite
semi-inner product.
Using the Cauchy-Schwarz inequality for positive-definite sesquilinear forms we observe that $K$ is a submodule of
$\m A\bigotimes_{alg} E_1$ where
$$K :=\left\{ \sum\limits_{i=1}^{n}a_{i}\otimes {x}
_{i}\in \mbox{$\m A\bigotimes_{alg} E_1$}:\sum\limits_{i,j=1}^{n}\left\langle {x}
_{i},\tau \left( \alpha ( a_{i}^{\ast }) a_{j}\right) {x}
_{j}\right\rangle =0\right\}.$$
Therefore $\langle\cdot,\cdot\rangle$ induces
naturally on the quotient module $\left({\m A\bigotimes_{alg} E_1}\right)/ K$ as a $\m B$-valued inner product. 
Henceforth we denote this induced inner-product by $\langle~,~\rangle$ itself. Assume $E_{0}$ to be
the Hilbert $\m{B}$-module obtained by the
completion of $\left(\m{A}\bigotimes _{{alg}}{E}_{1}\right)/K$. 

It is easy to check that $\left( E
_{0},\m B,U_{0}\right) $ is an S-module, where the 
unitary $U_{0}$ is defined by
\[
U_{0}\left( \sum\limits_{i=1}^{n} a_i\otimes {x_i} +K\right) =\sum\limits_{i=1}^{n}\alpha ( a_i)
\otimes U_{1}{x_i} +K~\mbox{where}~a\in \m A,{x}\in E_1.  
\]
Indeed, $U_0$ is a unitary, because for all $a,a'\in \m A$ and ${x},{y}\in E_1$ we get
\begin{eqnarray*}
 &&\left< U_{0}\left( \sum\limits_{i=1}^{n} a_i\otimes {x_i} +K\right),U_{0}\left( \sum\limits_{j=1}^{n} a_j\otimes {x_j} +K\right)\right>
 \\ &=&\sum\limits_{i,j=1}^{n}\langle \alpha \left( a_i\right)\otimes U_{1}{x_i} +K,\alpha \left( a_j\right)\otimes U_{1}{x_j} +K\rangle
=\sum\limits_{i,j=1}^{n}\langle U_1 {x_i},\tau(\alpha(\alpha(a_i)^*)\alpha(a_j))U_1 {x_j}\rangle
\\&=&\sum\limits_{i,j=1}^{n}\langle  {x_i},\tau(\alpha(a_i^*)a_j) {x_j}\rangle
= \left<  \sum\limits_{i=1}^{n} a_i\otimes {x_i} +K, \sum\limits_{j=1}^{n} a_j\otimes {x_j} +K\right>,
\end{eqnarray*}
and because $U_0$ is surjective. Since
\begin{eqnarray*}
 &&\left< U_{0}\left( \sum\limits_{i=1}^{n} a_i\otimes {x_i} +K\right),\sum\limits_{j=1}^{m} a'_j\otimes {y_j} +K\right>
 \\&=&\sum\limits_{i=1}^{n}\sum\limits_{j=1}^{m}\langle \alpha \left( a_i\right)
\otimes U_{1}{x_i} +K,a'_j\otimes {y_j} +K\rangle
=\sum\limits_{i=1}^{n}\sum\limits_{j=1}^{m}\langle U_1 {x_i},\tau(\alpha(\alpha(a_i)^*)a'_j) {y_j}\rangle
\\&=&\sum\limits_{i=1}^{n}\sum\limits_{j=1}^{m}\langle  {x_i},\tau(\alpha(a^*_i)\alpha^{-1}(a'_j)) U^*_1{y_j}\rangle
= \left< \sum\limits_{i=1}^{n} a_i\otimes {x_i} +K, \sum\limits_{j=1}^{m}\alpha^{-1}(a'_j)\otimes U^*_1{y_j} +K\right>,
\end{eqnarray*}
we obtain $U^*_{0} \left( \sum\limits_{j=1}^{m} a'_j\otimes {y_j} +K\right)=\sum\limits_{j=1}^{m}\alpha^{-1}(a'_j)\otimes U^*_1{y_j} +K$.
Define a map $V:{E}_{1}\rightarrow 
E_{0}$ by
\[
V{x}:=1\otimes U_{1}{x} +K~\mbox{where}~{x}\in E_1. 
\]
For each $x\in E_1$ we have
\begin{align*}
 \|Vx\|^2=\|\langle V{x},Vx\rangle\| &=\|\langle 1\otimes U_1{x}+K,1\otimes U_1{x}+K\rangle\|=\|\langle U_1{x},\tau(1)U_1{x}\rangle\|
 \\&\leq\|\tau(1)\|\|x\|^2.
\end{align*}
This implies that $V$ is bounded. For each $a_1,a_2,\ldots,a_n \in\m A$ and ${x},y_1,y_2,\ldots,y_n\in E_1$ we have
\begin{align}\label{eqn15}
 \left< V{x},\sum\limits_{i=1}^{n} a_i\otimes{y_i}+K\right> \nonumber &=\left< 1\otimes U_1{x}+K,\sum\limits_{i=1}^{n} a_i\otimes{y_i}+K\right>
 =\left< U_1{x},\sum\limits_{i=1}^{n}\tau(\alpha(1)a_i){y_i}\right>
 \\&=\left< {x},\sum\limits_{i=1}^{n} U^*_1\tau(a_i){y_i}\right>=\left< {x},\sum\limits_{i=1}^{n}\tau(a_i)U^*_1{y_i}\right>.
\end{align}
From Lemma 2.8 of \cite{HJU10} there exists $M>0$ such that $$(\tau(a_i^*)\tau(a_j))\leq M(\tau(\alpha(a_i^*)a_j)).$$
Indeed, for each $a_1,a_2,\ldots a_n\in\m A$ and $y_1,y_2,\ldots y_n\in E_1$ we get
\begin{align}\label{eqn16}
\left\|\sum_{i=1}^n\tau(a_i)U^*_1{y_i}\right\|^2 \nonumber &= \left\|\left< \sum_{i=1}^n\tau(a_i)U^*_1{y_i},\sum_{j=1}^n\tau(a_j)U^*_1{y_j}\right>\right\| 
\\ \nonumber&=\left\|\sum\limits_{i=1}^{n}\sum\limits_{j=1}^{n}\left< U^*_1{y_i}, \tau(a^*_i)\tau(a_j)U^*_1{y_j}\right> \right\|
\\ \nonumber&\leq M\left\|\sum_{i=1}^n\sum_{j=1}^n\left< U^*_1{y_i}, \tau(\alpha(a^*_i)a_j)U^*_1{y_j}\right>\right\|
\\ & =M \left\|\sum_{i=1}^n a_i\otimes y_i+K\right\|^2.
\end{align}
Therefore using Equation \ref{eqn15} and Inequality \ref{eqn16}, we conclude that $V$ is an adjointable map with adjoint
$$V^{\ast }\left(\sum\limits_{i=1}^{n}a_i\otimes {x_i} +K\right):=\sum\limits_{i=1}^{n}U^*_{1}\tau (a_i){x_i}
~\mbox{where $a_i\in\m A;~x_i \in {E}_{1}$ for $1\leq i\leq n$.}$$ 
For each $a_1,a_2,\ldots,a_n\in\m A;~x_1,x_2,\ldots,x_n \in {E}_{1}$ we obtain
\begin{align*}
 V^\natural\left(\sum\limits_{i=1}^{n}a_i\otimes{x_i}+K \right)&=U^*_{1}V^*U_{0}\left(\sum\limits_{i=1}^{n}a_i\otimes{x_i}+K \right)
 =U^*_{1}V^*\left(\sum\limits_{i=1}^{n}\alpha(a_i)\otimes U_1{x_i}+K \right)
 \\&=U^*_{1}\sum\limits_{i=1}^{n}\tau(\alpha(a_i))U^*_1 U_1{x_i}=U^*_{1}\sum\limits_{i=1}^{n}\tau(\alpha(a_i)){x_i}
 \\&=V^*\left(\sum\limits_{i=1}^{n}a_i\otimes{x_i}+K \right)
\end{align*}
which implies that $V^{\natural}=V^{\ast}$. Define the map $
\pi'_0:\m{A}\rightarrow \mathcal{B}^a(E_0)$ by
\begin{equation}
\pi'_0(a)\left(\sum\limits_{i=1}^n b_i\otimes {x_i} +{K}\right) =\sum\limits_{i=1}^n ab_i\otimes {x_i}
+{K}~  \label{m12}
\end{equation}
$\mbox{for all}~a,b_1,b_2,\ldots,b_n\in \m{A};x_1,x_2,\ldots,x_n \in {E}_{1}.$ We have
\begin{eqnarray*}
 &&\left\|\pi'_0(a)\left(\sum\limits_{i=1}^n a_i\otimes x_i+K\right)\right\|^2 = \left\|\sum\limits_{i=1}^n aa_i\otimes x_i+K\right\|^2
 \\ &=& \left\|\left< \sum\limits_{i=1}^n aa_i\otimes x_i+K, \sum\limits_{j=1}^n aa_j\otimes x_j+K\right>\right\|
 = \left\|\sum\limits_{i,j=1}^n\left<x_i,  \tau(\alpha(a^*_ia^*)aa_j) x_j\right>\right\|
 \\ &\leq & M(a)\left\|\sum\limits_{i,j=1}^n\left<x_i,  \tau(\alpha(a^*_i)a_j) x_j\right>\right\|
 = M(a)\left\|\left(\sum\limits_{i=1}^n a_i\otimes x_i+K\right)\right\|^2
\end{eqnarray*}
where $a,a_1,\ldots,a_n\in\m A$ and $x_1,\ldots,x_n\in E_1$. 
Thus for each $a\in\m A$, $\pi'_0(a)$ is a well-defined bounded linear operator from $E_0$ to $E_0$. Using
\begin{eqnarray*}
&&\left< \pi'_0(a)\left(\sum\limits_{i=1}^n a_i\otimes x_i+K\right), \sum\limits_{j=1}^m a'_j\otimes x'_j+K\right>
 \\&=&  \left<\sum\limits_{i=1}^n aa_i\otimes x_i+K, \sum\limits_{j=1}^m a'_j\otimes x'_j+K\right>
=  \sum\limits_{i=1}^n \sum\limits_{j=1}^m\left<x_i,  \tau(\alpha(a^*_ia^*)a'_j )x'_j\right>
\\ & = & \sum\limits_{i=1}^n \sum\limits_{j=1}^m\left<x_i,  \tau(\alpha(a^*_i)\alpha(a^*)a'_j) x'_j\right>
\\ &=& \left< \sum\limits_{i=1}^n  a_i\otimes x_i+K,  \sum\limits_{j=1}^m\alpha(a^*)a'_j \otimes x'_j+K\right>
\end{eqnarray*}
and
\begin{eqnarray*}
U_0\pi'_0(a^*)U^*_0\left(\sum\limits_{j=1}^m a'_j\otimes
x'_j+K\right) &=&
U_0\pi'_0(a^*)\left(\sum\limits_{j=1}^m\alpha^{-1}(a'_j)\otimes U^*_1
x'_j+K\right)
 \\ &=&  U_0\left(\sum\limits_{j=1}^m (a^*\alpha^{-1}(a'_j))\otimes U^*_1 x'_j+K\right)
\\ &=& \sum\limits_{j=1}^m \alpha(a^*)a'_j\otimes x'_j+K
\end{eqnarray*}
for all $a_1,\ldots,a_n,a'_1,\ldots,a'_m\in \m A$ and $x_1,\ldots,x_n,x'_1,\ldots,x'_m\in E_1$, it follows that 
$\pi'_0:\m{A}\rightarrow \mathcal{B}^a(E_0)$
is a well-defined map. Indeed, $\pi'_0:\m{A}\to \mathcal{B}^a(E_0)$ is an $U_0$-representation. Define an $U_0$-representation $\pi_0:\m A\to \mathcal{B}^a(E_0)$ by $\pi_0(a):=\pi'_0(\alpha(a))$ for all $a\in \m A$.
Since $V^{\natural}=V^{\ast },$ for all $a\in\m A$, ${x} \in {E}_{1}$ we obtain
\begin{equation*}
V^{\natural} \pi'_0(a)V{x} =V^{\ast }\left(
a\otimes U_{1} {x} +{K}\right) =U^*_{1}\tau
(a)U_{1}{x} =\tau (a){x}.
\end{equation*}
Therefore $\tau (a)=V^{\natural}\pi_0(a)V~\mbox{for all}~a\in \m{A}.$ Moreover, for each $x\in E_1$ and $a,b\in\m A$ we get
\begin{eqnarray*}
 V^*\pi'_0(a)^*\pi'_0(b)Vx &=& V^*U_0\pi'_0(a^*)U^{*}_0\pi'_0(b)Vx=V^*U_0\pi'_0(a^*)U^*_0(b\otimes U_1 x+ K)
 \\ &=& V^*U_0\pi'_0(a^*)(\alpha^{-1}(b)\otimes x+ K)
 \\ &=& V^*U_0(a^*\alpha^{-1}(b)\otimes x+ K)=V^*(\alpha(a^*\alpha^{-1}(b))\otimes U_1 x+ K)
 \\&=& U^*_1\tau(\alpha(a^*\alpha^{-1}(b)))U_1 x  =\tau(\alpha(a)^*b) x=V^*\pi'_0(\alpha(a)^*b)Vx.
\end{eqnarray*}
From this equality, it follows that 
\begin{eqnarray*}
  V^*\pi_0(a)^*\pi_0(b)V &=& V^*\pi'_0(\alpha(a))^*\pi'_0(\alpha(b))V=V^*\pi'_0(\alpha(\alpha(a))^*\alpha(b))V
  \\ &=& V^*\pi'_0(\alpha(\alpha(a)^*b))V= V^*\pi_0(\alpha(a)^*b)V
\end{eqnarray*}
for each $a,b\in\m A$.
\end{proof}

In the following theorem we extend the KSGNS construction for $\tau$-maps:

\begin{theorem}
\label{main} Assume $\m A$ to be a unital $C^*$-algebra and $\alpha:\m A\to\m A$ be an automorphism.
Suppose $\m B$ is a von Neumann algebra acting on a Hilbert space $\m H$ and ${E}$ is a Hilbert $\m{A}$-module.
Let $E_1$ be a Hilbert $\m B$-module and $E_2$ be a von Neumann $\m B$-module such that $(E_1,\m B,U_1)$ and $(E_2,\m B,U_2=%
\mathrm{id}_{E_2})$ be S-modules. If $\tau:%
\m{A}\rightarrow\mathcal{B}^a({E}_{1})$ is an $\alpha $-CP map and  $T:%
{E}\rightarrow \mathcal{B}^a({E}_{1},{E}_{2})$ is a $\tau$-map, then there exist

\begin{itemize}
\item [(i)]
\begin{enumerate}
\item [(a)] a von Neumann $\m B$-module $E_3$ with a unitary $U_3$ such that $\left( E%
_{3},\m B,U_{3}\right)$ is an S-module,
\item [(b)] an
$U_3$-representation $\pi$ of $\m{A}$ on $\left( E%
_{3},\m B,U_{3}\right)$ with a map  $V\in \ma B^a (E_1,E_3)$ such
that $V^{\natural}=V^{\ast }$,
and
$$\tau(a)=V^{*}\pi(a)V~\mbox{for all}~a\in\m A,~$$
\end{enumerate}
\item [(ii)]
\begin{enumerate}
\item [(a)] a von Neumann $\m B$-module $E_4$ which is an S-module $(E_4,\m B,U_4=id_{E_4})$ and a map $\Psi:E\to \ma B^a (E_3,E_4)$ which is a 
$\pi$-map,
\item [(b)] a coisometry $W$ from $ E_2$ onto $ E_4$ satisfying $W^{\natural}=W^{\ast }$,
 $$T(x)=W^{*}\Psi(x)V~\mbox{for all}~x\in E.$$
\end{enumerate}
\end{itemize}
\end{theorem}

\begin{proof}
By Theorem \ref{main0} we obtain the triple $(\pi_0,V, E_0)$
associated to $\tau$ where $(E_0,\m B, U_0)$ is an S-module. Here the Hilbert $\m B$-module $E_0$ satisfies 
$\overline{\rm span}~{\pi_0(\m A)VE_1}=E_0$,
 $V\in \ma B^a (E_1,E_0)$, and $\pi_0$ is an $U_0$-representation of $\m A$ to $\ma B^a (E_0)$ such that
$$\tau(a)=V^{*}\pi_0(a)V~\mbox{for all}~a\in\m A.~$$
We obtain a von Neumann $\m B$-module $E_3$ by taking the strong operator topology closure of
$E_0$ in $\ma B(\m H,E_0\bigotimes\m  H)$. Consider the element of $ \ma B^a(E_1, E_3)$ which gives the 
same value as $V$ when evaluated on the elements of $E_1$, because $E_0$ is canonically embedded in $E_3$. 
We denote this element of $ \ma B^a(E_1, E_3)$ by $V$. 
Fix $\displaystyle\lim_{\alpha} x^0_{\alpha}\in E_3$ with $x^0_{\alpha}\in E_0$. It is easy to check that 
$\mbox{sot-}\displaystyle\lim_{\alpha} \pi_0(a)x^0_{\alpha}$ exists for each $a\in \m A$.
The $U_0$-representation
$\pi_0:\m A\to \ma B^a(E_0)$ extends to a representation of $\m A$ on
$E_3$ as follows:
\[
 \pi(a)(x):=\mbox{sot-}\displaystyle\lim_{\alpha} \pi_0(a)x^0_{\alpha}~\mbox{where $a\in \m A$,
 $x$=sot-$\displaystyle\lim_{\alpha} x^0_{\alpha}\in E_3$ with $x^0_{\alpha}\in E_0$.}
\]
For each $a\in \m A$, $x$=sot-$\displaystyle\lim_{\alpha} x^0_{\alpha}$ and
$y$=sot-$\displaystyle\lim_{\beta} y^0_{\beta}\in E_3$ with
$x^0_{\alpha},y^0_{\beta}\in E_0$ we have
\begin{align*}
 \langle \pi(a) x,y\rangle &=\mbox{sot-}\displaystyle\lim_{\beta}\langle \pi(a) x,y^0_{\beta}\rangle
 =\mbox{sot-}\displaystyle\lim_{\beta}(\mbox{sot-}\displaystyle\lim_{\alpha}\langle y^0_{\beta},\pi_0(a) x^0_{\alpha}\rangle)^*
 \\ &=\mbox{sot-}\displaystyle\lim_{\beta}(\mbox{sot-}\displaystyle\lim_{\alpha}\langle \pi_0(a)^*y^0_{\beta}, x^0_{\alpha}\rangle)^*
 = \langle  x,\pi(a^*)y\rangle,
\end{align*}
i.e., $\pi(a)\in \ma B^a(E_3)$ for each $a\in\m A$. Let $U_3: E_3\to E_3$ be a map defined by
\[
 U_3(x):=\mbox{sot-}\displaystyle\lim_{\alpha} U_0(x^0_{\alpha}))~\mbox{where
 $x$=sot-$\displaystyle\lim_{\alpha} x^0_{\alpha}\in E_3$ with $x^0_{\alpha}\in E_0$.}
\]
It is easy to observe that $U_3$ is a unitary, $(E_3,\m B, U_3)$ is an S-module and the triple $(\pi, V, E_3)$ 
satisfies all the conditions of the statement
(i).

Let $E'_4$ be the Hilbert $\ma{B}$-module $\overline{\rm span}~ T \left( {E}\right) {E}_{1}%
$. For each $x\in {E}$, define a map $\Psi_0( x)
:E_0\rightarrow E'_4$
 by
\begin{equation}
\Psi_0\left( x\right) \left( \sum\limits_{i=1}^{n}\pi_0\left(
a_{i}\right) V{x} _{i}\right) =\sum\limits_{i=1}^{n}T \left(
xa_{i}\right) {x}_{i}  \label{m5}
\end{equation}%
for all $a_1,a_2,\ldots,a_n\in\m A$ and $x_1,x_2,\ldots,x_n\in E_1$.
Each $\Psi_0(x)$ is a bounded right $\m B$-linear map from $E_0$ to $E'_4$. Indeed, we have
\begin{align}\label{eqn6}\nonumber
& \left<\Psi_0\left( x\right) \left( \sum\limits_{i=1}^{n}\pi_0\left(
a_{i}\right) V{x} _{i}\right),\Psi_0\left( x\right) \left( \sum\limits_{j=1}^{m}\pi_0\left(
a'_{j}\right) V{x}'_{j}\right)\right>
\\&\nonumber=\left< \sum\limits_{i=1}^{n}T \left( xa_{i}\right) {x}
_{i},\sum\limits_{j=1}^{m}T \left( ya'_{j}\right) {x}' _{j}\right>
= \sum\limits_{i=1}^{n}\sum\limits_{j=1}^{m}\left< {x}
_{i},T \left( xa_{i}\right) ^*T \left( ya'_{j}\right)
{x}' _{j}
\right>
\\& \nonumber= \sum\limits_{i=1}^{n}\sum\limits_{j=1}^{m}\left< {x}
_{i},\tau(\langle  xa_{i}, ya'_{j}\rangle)
{x}' _{j}
\right>
= \sum\limits_{i=1}^{n}\sum\limits_{j=1}^{m}\left\langle {x}
_{i},V^*\pi_0\left( a^*_i \left\langle x,y\right\rangle a'_j\right) V  {x}'
_{j}\right\rangle\\
& \nonumber= \sum\limits_{i=1}^{n}\sum\limits_{j=1}^{m}\left\langle {x}
_{i},V^*\pi_0( a_i)^* \pi_0(\left\langle x,y\right\rangle a'_j) V  {x}'
_{j}\right\rangle
\\
& = \left\langle \sum\limits_{i=1}^{n}\pi_0\left( a_{i}\right) V{x}
_{i},\pi_0\left( \left\langle x,y\right\rangle \right)
\sum\limits_{j=1}^{m}\pi_0\left( a'_{j}\right) V{x}'
_{j}\right\rangle
\end{align}%
for all $x,y\in E$; and 
$a_i,a'_j\in\m A$ and
$x_i,x'_j\in E_1$ for $1\leq i\leq n,$ $1\leq j\leq m$.  We
denote by $E_4$ the strong operator topology closure of ${E}'_{4}$ in
$\ma B(\m H, {E}'_{4}\bigotimes \m H)$. For $x\in E$ define a mapping $\Psi(x):E_3\to
E_4$ by
\[
\Psi(x)(z):=\mbox{sot-}\lim_{\alpha} \Psi_0 (x)
z^0_{\alpha}~\mbox{where $z$=sot-$\displaystyle\lim_{\alpha}
z^0_{\alpha}\in E_3$ for $z^0_{\alpha}\in E_0$.}~
\]
Note that the limit $\mbox{sot-}\lim_{\alpha} \Psi_0 (x)
z^0_{\alpha}$ exists. For all
$z$=sot-$\displaystyle\lim_{\alpha} z^0_{\alpha}\in E_3$ with $z^0_{\alpha}
\in E_0$ and $x,y\in E$ we have
\begin{align*}
\langle \Psi(x) z,\Psi(y) z\rangle &=
\mbox{sot-}\lim_{\alpha}\{\mbox{sot-}\lim_{\beta}\langle
\Psi_0(y)z^0_{\alpha}, \Psi_0(x)z^0_{\beta}\rangle\}^*
\\ &=\mbox{sot-}\lim_{\alpha}\{\mbox{sot-}\lim_{\beta}\langle
z^0_{\alpha}, \pi_0(\langle x,y\rangle)z^0_{\beta}\rangle\}^*
=\langle
z,\pi(\langle x,y\rangle)z\rangle.
\end{align*}
Since $E_3$ is a von Neumann $\m B$-module, we conclude that
$\Psi:E\to \ma B^a (E_3, E_4)$ is a $\pi$-map. Because $E_4$
is a von Neumann $\m B$-submodule of $ E_2$, we get an
orthogonal projection
 from $ E_2$ onto $ E_4$ (cf. Theorem 5.2 of \cite{Sk00}) which we denote by $W$. 
 Therefore $W^*$ is the inclusion map from $ E_4$ to $ E_2$, and hence
 $WW^*=id_{E_4}$, i.e., $W$ is a coisometry. Considering $E_4$ 
 as an S-module 
 $(E_4,\m B,U_4=id_{E_4})$ it is evident that $W^\natural (x)=U^*_2W^*U_4(x)=W^*(x)$ for all $x\in E_2$.
 Eventually
\begin{eqnarray*}
W^*\Psi(x) V=\Psi(x)V=\Psi(x)(\pi(1)V)=T(x)~\mbox{for all}~x\in
E.\qedhere
\end{eqnarray*}
\end{proof}

\section{A partial factorization theorem for $\mf{K}$-families}\label{secK3}

Suppose $\m B$ and $\m C$ are unital $C^*$-algebras. We denote the set of all bounded linear maps from  $\m B$ to $\m C$ by $\ma
B(\m B,\m C)$. Let
$\alpha$ be an automorphism on $\m B$, i.e., $\alpha:\m B\to \m B$ is a bijective unital $*$-homomorphism. For a set $\Omega$, by a {\it kernel} 
$\mf K$ over $\Omega$ from $\m B$ to $\m C$ we mean a function $\mf{K}:\Omega\times \Omega \to \ma B(\m B,\m C)$, 
and $\mf K$ is called {\it Hermitian} if $\mf{K}^{\sigma, \sigma'}(b^*)=\mf{K}^{\sigma', \sigma}(b)^*$ for all $\sigma,\sigma'\in \Omega$ and $b\in\m B$.
We say that a Hermitian kernel $\mf K$ over $\Omega$ from $\m B$ to $\m C$ is an {\it $\alpha$-completely positive definite kernel} or
an {\it $\alpha$-CPD-kernel} over $\Omega$ from $\m B$ to $\m C$ if for finite choices $\sigma_i\in \Omega$, $b_i\in \m B$, $c_i\in \m C$ we have
\begin{itemize}
 \item [(i)] $ \sum_{i,j} c^*_i \mf{K}^{\sigma_i, \sigma_j} (\alpha(b_i)^* b_j) c_j \geq 0, $
 \item [(ii)] $ \mf{K}^{\sigma_i, \sigma_j} (\alpha(b))= 
         \mf{K}^{\sigma_i, \sigma_j}(b) $ for all $b\in\m B$,
 \item [(iii)] for each $b\in\m B$ there exists $M(b)>0$ such that 
\begin{eqnarray*}
\left\|\sum\limits_{i,j=1}^n c^*_i  \mf{K}^{\sigma_i,\sigma_j}(\alpha(b^*_ib^*)bb_j) c_j\right\|
 &\leq & M(b)\left\|\sum\limits_{i,j=1}^n   c^*_i \mf{K}^{\sigma_i,\sigma_j}(\alpha(b^*_i)b_j) c_j\right\|.
\end{eqnarray*}
\end{itemize}

In this section we discuss the decomposition of $\mf K$-families for an $\alpha$-CPD kernel in terms of 
reproducing kernel S-correspondences which is defined as follows:

\begin{definition}
 Let $\m A$ and $\m{B}$ be unital $C^{\ast}$-algebras. An S-module $(\m F,\m B,U)$ is called an {\rm S-correspondence over $\Omega$ from $\m A$ 
 to $\m B$} if there exists a $U$-representation $\pi$ of $\m{A}$ on $(\m F,\m B,U)$, i.e., $\m F$ is also a left $\m A$-module with 
 \[
  af:=\pi(a)f~\mbox{for all}~a\in\m A,f\in \m F.
 \]
 Let $\Omega$ be a set. If $(\m F,\m B,U)$ is an S-correspondence from $\m A$ to $\m B$, consisting of functions from $\Omega\times\m A$ to $\m B$, 
 which forms a vector space with point-wise vector space operations, and for each $\sigma\in \Omega$ there exists an element $k_{\sigma}$ in
 $\m F$ called {\rm kernel element} satisfying 
 \[
  f(\sigma,a)=\langle k_{\sigma},af\rangle~\mbox{for all}~a\in\m A,~f\in \m F,
 \]
 then this S-correspondence is called a {\rm reproducing kernel S-correspondence} over $\Omega$ from $\m A$ to $\m B$.
The mapping $\mf K:\Omega\times \Omega\to \ma B(\m A,\m B)$ defined by
\[
 \mf K^{\sigma,\sigma'}(a)=k_{\sigma'}(\sigma,a)~\mbox{for all}~a\in\m A,~\sigma'\in \Omega 
\]
is called the {\rm reproducing kernel} for the reproducing kernel S-correspondence.
\end{definition}

In Theorem 3.1 of \cite{TDT13}, Bhattacharyya, Dritschel and Todd proved that a kernel $\mf{K}$ is dominated by a CPD-kernel 
if and only if  $\mf{K}$ has a Kolmogorov decomposition in which the module forms a Krein $C^*$-correspondence.  
Skeide's factorization theorem for $\tau$-maps \cite{Sk12} is based on the Paschke's GNS construction (cf. Theorem 5.2, \cite{Pas73}) for 
CP map $\tau$. Using the Kolmogorov decomposition we proved
a factorization theorem for $\mf{K}$-families in Theorem 2.2 of \cite{DH14} when $\mf{K}$ is a CPD-kernel. In Theorem 3.5 of \cite{BBQH09}, 
a characterization of a CPD-kernel in terms of reproducing kernel $C^*$-correspondences was obtained (also see Theorem 3.2, \cite{H08}). 

\begin{theorem}\label{fact}
  Let $ \mf{K}$ be a Hermitian kernel over a set
$\Omega$ from a unital $C^*$-algebra $\m B$ to a unital $C^*$-algebra $\m C$. Assume $\alpha$ to be an automorphism on $\m B$. 
 Let $\ma K^{\sigma}$ be a map from Hilbert $\m B$-module $E$ to Hilbert $\m C$-module $F$, for each $\sigma\in \Omega$. 
 Then the following statements are equivalent:
 \begin{itemize}
  \item [(i)] The family $\{\ma K^{\sigma}\}_{\sigma\in \Omega}$ is a $\mf{K}$-family where $\mf{K}$ is an $\alpha$-CPD-kernel.
  \item [(ii)] $\mf{K}$ is the reproducing kernel for an reproducing kernel S-correspondence
  $\m F=\m F(\mf K)$ over $\Omega$ from $\m B$ to $\m C$, i.e., there is an S-correspondence
  $\m F=\m F(\mf K)$ whose elements are $\m C$-valued functions on $\Omega\times\m B$ such that the function 
  $$k_{\sigma'}(\sigma,b):=\mf K^{\sigma,\sigma'}(b)\in\m F(\mf K)~\mbox{for all}~\sigma,\sigma'\in \Omega;~b\in\m B$$ and 
  has the reproducing property 
  \[
   \langle k_{\sigma},bf\rangle=\langle \alpha(b^*)k_{\sigma},f\rangle=f(\sigma,b)~\mbox{for all}~\sigma\in \Omega,~f\in \m F(\mf K),~b\in\m B
  \]
where $bk_{\sigma}\in \m F$ is given by
\[
 (bk_{\sigma})(\sigma',b'):=\mf K^{\sigma,\sigma'}(b'b)~\mbox{for all}~b'\in \m B.
\]
 Further,  \begin{eqnarray}\label{eqn13}
\langle \ma K^\sigma(xb)c,\ma K^{\sigma'}(x'b')c'\rangle
=  \langle \alpha( b){k_{\sigma}}c , \langle x, x'\rangle b'{k_{\sigma'}} c' \rangle
 \end{eqnarray}
for each $b,b'\in \m B;~c,c'\in \m C;~x,x'\in E;~\sigma,\sigma'\in \Omega$.
 \end{itemize}
 \end{theorem}
\begin{proof}

Suppose (ii) is given.  Thus from the reproducing property it follows that 
\begin{align*}
  \sum_{i,j} c^*_i \mf{K}^{\sigma_i, \sigma_j} (\alpha(b^*_i) b_j) c_j
 &=\sum_{i,j} c^*_i k_{\sigma_j}(\sigma_i, \alpha(b^*_i) b_j)c_j=\sum_{i,j} c^*_i \langle k_{\sigma_i}, \alpha(b^*_i) b_j
k_{\sigma_j}\rangle  c_j  
\\ &= \left<\sum_{i}b_ik_{\sigma_i}c_i
, \sum_{j} b_j k_{\sigma_j}c_j\right>
 \geq 0
\end{align*}
for all finite choices of $\sigma_i\in \Omega$, $b_i\in \m B$, $c_i\in \m C$. Further, for all $b\in\m B$ and $\sigma,\sigma'\in \Omega$ we get
\begin{align*}
 \mf{K}^{\sigma, \sigma'}(\alpha(b))
 &=k_{\sigma'}(\sigma,\alpha(b))
 =\langle k_{\sigma},\alpha(b) k_{\sigma'}\rangle=\langle b^* k_{\sigma},k_{\sigma'}\rangle
\\ &=(\langle k_{\sigma'},b^* k_{\sigma}\rangle)^*=k_{\sigma}(\sigma',b^*)^*=\mf{K}^{\sigma', \sigma}(b^*)^*=\mf{K}^{\sigma, \sigma'}(b).
\end{align*}
Finally, for a fixed $b\in \m B$ and each finite choices $\sigma_i\in \Omega$, $b_i\in \m B$, $c_i\in \m C$ we obtain
\begin{align*}
 &\left\|\sum\limits_{i,j=1}^n c^*_i  \mf{K}^{\sigma_i,\sigma_j}(\alpha(b^*_ib^*)bb_j) c_j\right\|
 =\left\|\sum\limits_{i,j=1}^n c^*_i  k_{\sigma_j}(\sigma_i,\alpha(b^*_ib^*)bb_j) c_j\right\|
  \\&=\left\|\sum\limits_{i,j=1}^n c^*_i  \langle k_{\sigma_i},(\alpha(b^*_ib^*)bb_j) k_{\sigma_j}\rangle c_j\right\|
=\left\|\left< \sum\limits_{i=1}^n  b_ik_{\sigma_i}c_i,\alpha(b)^*b\left(\sum\limits_{j=1}^n b_j k_{\sigma_j}c_j\right)\right> \right\|
     \\&\leq\|\alpha(b)^*b\|\left\| \sum\limits_{i=1}^n  b_ik_{\sigma_i}c_i \right\|^2
       \leq\|b\|^2\left\|\left< \sum\limits_{i=1}^n  b_ik_{\sigma_i}c_i,\sum\limits_{j=1}^n b_j k_{\sigma_j}c_j\right> \right\|
       \\&=\|b\|^2\left\|\sum\limits_{i,j=1}^n   c^*_i \mf{K}^{\sigma_i,\sigma_j}(\alpha(b^*_i)b_j) c_j\right\|.
\end{align*}
Thus the function $\mf{K}$ is an $\alpha$-CPD-kernel.
On the other hand, for each $x,x' \in E;~\sigma,\sigma'\in \Omega$ we obtain
   \begin{align*}
    \langle \ma K^{\sigma} (x),\ma K^{{\sigma}'}(x')\rangle
&=\langle  k_{\sigma},(\langle x,x'\rangle)k_{\sigma'} \rangle
      \\&= k_{\sigma'}(\sigma, \langle x,x'\rangle)
      =\mf{K}^{\sigma,\sigma'}(\langle x, x'\rangle).
\end{align*}
So $\{\ma K^{\sigma}\}_{\sigma\in \Omega}$ is a $\mf{K}$-family, i.e., (i) holds.

Conversely, suppose (i) is given. For each $\sigma'\in \Omega$ let ${k_{\sigma'}}:\Omega\times\m B\to\m C$
be a map defined by ${k_{\sigma'}}(\sigma,b):=\mf{K}^{\sigma,\sigma'}(b)$ where $\sigma\in \Omega,~b\in\m B$. Let us define the mapping 
$b{k_{\sigma'}}$ by
$(\sigma,b')\mapsto \mf{K}^{\sigma,\sigma'}(b'b)=k_{\sigma'}(\sigma,b'b)$  where $\sigma,\sigma'\in \Omega$ and $b,b'\in\m B$. For fixed $c\in \m C$ 
we define the function 
${k_{\sigma'}}c$ by
 $(\sigma,b)\mapsto \mf{K}^{\sigma,\sigma'}(b)c=k_{\sigma'}(\sigma,b)c$ for all $\sigma,\sigma'\in \Omega$ and $b\in\m B$. 
 In a canonical way define $(bk_{\sigma})c$ and $b(k_{\sigma}c)$ for all $\sigma\in \Omega,$ $b\in\m B,$ and $c\in\m C$.  
 Let $\m F_0$ be the right $\m C$-module generated by the set $\{bk_{\sigma}:b\in\m B,~\sigma\in \Omega\}$ consisting of $\m C$-valued
 functions on $\Omega\times\m B$, i.e.,  
 $\m F_0=\{\sum^m_{j=1} (b_jk_{\sigma_j}) c_j:b_1,\ldots,b_m\in\m B; 
 c_1,\ldots,c_m\in\m C;\sigma_1,\ldots,\sigma_m\in \Omega;m\in\mathbb{N} \}.$
 Note that $(bk_{\sigma})c=b(k_{\sigma}c)$ for all $\sigma\in \Omega,$ $b\in\m B,$ and $c\in\m C$ and hence we write
 $\m F_0=\{\sum^m_{j=1} b_jk_{\sigma_j} c_j:b_1,\ldots,b_m\in\m B; 
  c_1,\ldots,c_m\in\m C;\sigma_1,\ldots,\sigma_m\in \Omega;m\in\mathbb{N} \}.$ Define a map $\langle\cdot,\cdot\rangle:\m F_0\times\m F_0\to\m C$ by
\begin{align}\label{eqn12}
  \left<f ,g \right>:=
  \sum^m_{j=1}\sum^n_{i=1} c^*_j\mf{K}^{\sigma_j,\sigma^{\prime}_i}(\alpha(b_j)^*b'_i)c'_i
\end{align}
where $f=\sum^m_{j=1} b_jk_{\sigma_j} c_j$, $g=\sum^n_{i=1} b'_ik_{\sigma'_i} c'_i\in \m F_0$.
Further with $f=\sum^m_{j=1}  b_jk_{\sigma_j} c_j$ and $g=\sum^n_{i=1} b'_ik_{\sigma'_i} c'_i$ in $\m F_0$, we obtain
\begin{eqnarray}\label{eqn11}\nonumber   
        &&\sum^m_{j=1} c^*_jg(\sigma_j,\alpha(b_j)^*)            
    \\&=& \sum^m_{j=1}\sum^n_{i=1} c^*_jb'_ik_{\sigma^{\prime}_i}(\sigma_j,\alpha(b_j)^*)c'_i
    = \sum^m_{j=1}\sum^n_{i=1} c^*_jk_{\sigma^{\prime}_i}(\sigma_j,\alpha(b_j)^*b'_i)c'_i
   \nonumber \\ &=&\sum^m_{j=1}\sum^n_{i=1} c^*_j\mf{K}^{\sigma_j,\sigma^{\prime}_i}(\alpha(b_j)^*b'_i)c'_i=\sum^m_{j=1}\sum^n_{i=1} c^*_j\mf{K}^{\sigma_j,\sigma^{\prime}_i}(b_j^*\alpha^{-1}(b'_i))c'_i
 \nonumber \\& =& \sum^m_{j=1} c^*_j\sum^n_{i=1} \mf{K}^{\sigma^{\prime}_i,\sigma_j}(\alpha^{-1}(b^{\prime}_i)^*b_j)^*c^{\prime}_i
  \nonumber= \sum^m_{j=1} c^*_j\sum^n_{i=1} k_{\sigma_j}(\sigma^{\prime}_i,\alpha^{-1}(b^{\prime}_i)^*b_j)^*c^{\prime}_i
      \\ \nonumber&=& \sum^m_{j=1} c^*_j\sum^n_{i=1} (b_jk_{\sigma_j}(\sigma^{\prime}_i,\alpha^{-1}(b^{\prime}_i)^*))^*c^{\prime}_i
           \nonumber=\sum^n_{i=1} \left(\sum^m_{j=1}  b_j k_{\sigma_j}(\sigma^{\prime}_i,\alpha^{-1}(b^{\prime}_i)^*)c_j\right)c^{\prime}_i
                        \\&=&\sum^n_{i=1}f(\sigma^{\prime}_i,\alpha^{-1}(b^{\prime}_i)^*)c^{\prime}_i.
  \end{eqnarray}
Thus the function $\langle\cdot,\cdot\rangle$ defined above does not depend on the representations chosen for $f$ and $g$.
Since $\mf{K}$ is an $\alpha$-CPD-kernel, 
\[
  \left<\sum^m_{j=1} b_jk_{\sigma_j} c_j,  \sum^m_{i=1} b_ik_{\sigma_i} c_i\right>=
  \sum^m_{=1}\sum^m_{i=1} c^*_j\mf{K}^{\sigma_j,\sigma_i}(\alpha(b_j)^*b_i)c_i\geq 0.
 \]
Therefore the map $\langle\cdot,\cdot\rangle$ is positive definite. 
For $f:=\sum^m_{j=1} b_jk_{\sigma_j} c_j$ $\in\m F_0,$ $b\in\m B,~ c\in\m C$ and $\sigma\in \Omega$,  Equations \ref{eqn12} and \ref{eqn11}, 
and the Cauchy-Schwarz 
inequality gives 
\[\|f(\sigma,b)c\|^2=\|\langle f,\alpha(b)^*{k_{\sigma}}c\rangle\|^2\leq \|\langle\alpha(b)^*{k_{\sigma}}c,\alpha(b)^*{k_{\sigma}}c\rangle\|\|
\langle f,f\rangle\| .\]
So $f\in \m F_0$ vanishes pointwise if $\langle f,f\rangle=0$. This implies that $\m F_0$ is a right inner-product $\m C$-module 
with respect to $\langle\cdot,\cdot\rangle$. 
Let $\m F$ be the completion of $\m F_0$. 
It is easy to observe that the linear map $f\mapsto ((\sigma,b)\mapsto\langle\alpha(b^*)k_{\sigma},f\rangle)$, from $\m F$ to the set of all 
functions from $\Omega\times\m B$ to $\m C$, is injective. Therefore we identify $\m F$ as a subspace of the set of all 
functions from $\Omega\times\m B$ to $\m C$.

If $\sum^m_{j=1}  b_jk_{\sigma_j} c_j$ and $\sum^n_{i=1} b'_ik_{\sigma'_i} c'_i$ are elements of $\m F_0$, then we get 
\begin{eqnarray}\label{eqn14}
 \left<\sum^m_{j=1}  b_jk_{\sigma_j} c_j, \sum^n_{i=1} b'_ik_{\sigma'_i} c'_i\right>
 \nonumber &=&\sum^m_{j=1}\sum^n_{i=1}c^*_j\mf{K}^{\sigma_j,\sigma'_i} (\alpha(b_j)^*b'_i)c'_i
= \sum^m_{j=1}\sum^n_{i=1}c^*_j\mf{K}^{\sigma_j,\sigma'_i} (\alpha(\alpha(b^*_j)b'_i))c'_i
  \\&=& \left< \sum^m_{j=1}  \alpha(b_j)k_{\sigma_j} c_j, \sum^n_{i=1} \alpha(b'_i)k_{\sigma'_i}c'_i\right>.
\end{eqnarray}
Therefore we get an isometry $U:\m F\to\m F$ by $\sum^n_{i=1} b_i{k_{\sigma_i}}c_i\mapsto \sum^n_{i=1}\alpha(b_i){k_{\sigma_i}}c_i.$ Moreover, 
from Equation \ref{eqn14}, it is easy to check that $U$ is a unitary with the adjoint  $U^*:\m F\to\m F$ defined by 
$\sum^n_{i=1} b_i{k_{\sigma_i}}c_i\mapsto \sum^n_{i=1}\alpha^{-1}(b_i){k_{\sigma_i}}c_i.$ We define a sesquilinear form
$[\cdot,\cdot]:\m F\times\m F\to\m C$ as follows:
\[ [f,f']:=\langle f,Uf'\rangle
 \]
where $f,f'\in \m F$. Indeed, for $\sum^m_{j=1} b_jk_{\sigma_j} c_j$, $\sum^n_{i=1} b'_ik_{\sigma'_i} c'_i\in \m F$ we obtain

\begin{eqnarray*}
 \left[\sum^m_{j=1} b_jk_{\sigma_j}c_j ,\sum^n_{i=1} b'_ik_{\sigma'_i} c'_i\right] =  \left<\sum^m_{j=1} b_j k_{\sigma_j} c_j,\sum^n_{i=1} 
 \alpha(b'_i)k_{\sigma'_i} c'_i\right> .
\end{eqnarray*}
For each $b\in \m B$ define $\pi(b):\m F\to\m F$ by 
$$\pi(b)\left(\sum^m_{j=1}b_j{k_{\sigma_j}} c_j\right):=\sum^m_{j=1}bb_j{k_{\sigma_j}} c_j~\mbox{for all}~b'\in \m B,\sigma\in \Omega,c\in\m C.$$ 
Therefore for $b,b_1,\ldots,b_n\in\m B$; $c_1,\ldots,c_n\in \m C$ and $\sigma_1,\ldots,\sigma_n\in \Omega$ we have
\begin{eqnarray*}
 \left\|\pi(b)\left(\sum\limits_{i=1}^n b_i{k_{\sigma_i}} c_i\right)\right\|^2 
 &=& \left\|\sum\limits_{i=1}^n bb_i{k_{\sigma_i}} c_i\right\|^2
 = \left\|\left< \sum\limits_{i=1}^n bb_i{k_{\sigma_i}} c_i, \sum\limits_{j=1}^n bb_jk_{\sigma_j} c_j\right>\right\|
  \\&=& \left\|\sum\limits_{i,j=1}^n c^*_i  \mf{K}^{\sigma_i,\sigma_j}(\alpha(b^*_ib^*)bb_j) c_j\right\|
 \\&\leq & M(b)\left\|\sum\limits_{i,j=1}^n   c^*_i \mf{K}^{\sigma_i,\sigma_j}(\alpha(b^*_i)b_j) c_j\right\|
= M(b)\left\|\sum\limits_{i=1}^n b_i{k_{\sigma_i}} c_i\right\|^2.
\end{eqnarray*}
This implies that for each $b\in\m B$, $\pi(b)$ is a well defined bounded linear operator from $\m F$ to $\m F$. From
\begin{eqnarray*}
 &&\left< \pi(b)\left(\sum\limits_{i=1}^n b_i{k_{\sigma_i}} c_i\right), \sum\limits_{j=1}^m b'_j{k_{\sigma'_j}} c'_j\right>
=  \left< \sum\limits_{i=1}^n bb_i{k_{\sigma_i}} c_i, \sum\limits_{j=1}^m b'_j{k_{\sigma'_j}} c'_j\right>
\\&=&  \sum\limits_{i=1}^n \sum\limits_{j=1}^m c^*_i  \mf{K}^{\sigma_i,\sigma'_j}(\alpha(b^*_ib^*)b'_j )c'_j
=  \sum\limits_{i=1}^n \sum\limits_{j=1}^m c^*_i  \mf{K}^{\sigma_i,\sigma'_j}(\alpha(b^*_i)\alpha(b^*)b'_j )c'_j
\\ &=&  \left< \sum\limits_{i=1}^n b_i{k_{\sigma_i}} c_i, \sum\limits_{j=1}^m \alpha(b^*)b'_j{k_{\sigma'_j}}c'_j \right>
\end{eqnarray*}
and
\begin{eqnarray*}
U\pi(b^*)U^*\left(\sum\limits_{j=1}^m b'_j{k_{\sigma'_j}} c'_j\right) &=&
U\pi(b^*)\left(\sum\limits_{j=1}^m \alpha^{-1}(b'_j){k_{\sigma'_j}}c'_j \right)
=  U\left(\sum\limits_{j=1}^m b^{*}\alpha^{-1}(b'_j){k_{\sigma'_j}} c'_j\right)
\\ &=& \left(\sum\limits_{j=1}^m \alpha(b^{*}\alpha^{-1}(b'_j)){k_{\sigma'_j}}c'_j )\right)
=\left(\sum\limits_{j=1}^m \alpha(b^*) b'_j{k_{\sigma'_j}} c'_j\right)
\end{eqnarray*}
for all $b,b'_1,\ldots,b'_n\in\m B$; $c'_1,\ldots,c'_n\in \m C$ and $\sigma'_1,\ldots,\sigma'_n\in \Omega$, it follows that 
$\pi$ is an $U$-representation from $\m B$ to the S-module $(\m F,\m C,U)$ and 
$\m F$ becomes an S-correspondence with left action induced by $\pi$.
Indeed, using Equations \ref{eqn12} and \ref{eqn11}, 
we can realize elements $g$ of $\m F$ as $\m C$-valued 
functions on $\Omega\times \m B$ such that they satisfy the following reproducing property:
\[
 g(\sigma,b)=\langle k_{\sigma},bg\rangle~\mbox{for all}~\sigma\in \Omega,~b\in\m B.
\]

Eventually, for each $b,b'\in \m B;~c,c'\in\m C;~x,x'\in E;~\sigma,\sigma'\in \Omega$ we get
\begin{eqnarray*}
\langle \ma K^\sigma(xb)c,\ma K^{\sigma'}(x'b')c'\rangle
&=& c^*\mf{K}^{\sigma,\sigma'}(\langle xb,x'b'\rangle)c'
\\&=&  \langle  {k_{\sigma}} c, b^*\langle x, x'\rangle b'{k_{\sigma'}}  c'\rangle 
\\&=&  \langle \alpha( b){k_{\sigma}} c, \langle x, x'\rangle b'{k_{\sigma'}}  c'\rangle .
\qedhere
 \end{eqnarray*} 
\end{proof}
\begin{remark} The $U$-representation $\pi$ in Theorem \ref{fact} is not necessarily $*$-preserving, 
and $\pi(b^*)^*=\pi(\alpha(b))$ for all $b\in\m B$. In addition if we assume $\alpha=id_{\m B}$ in Theorem \ref{fact}
(i.e., $\mf{K}$ is a CPD-kernel), then $U$ is the identity map and $\pi$ becomes a $*$-preserving representation, and hence
 $\m F$ is a $C^*$-correspondence. Define a linear map
$\nu$ from the interior tensor product $E\bigotimes_{\m B} \m F$ to $F$ by
$$\nu(x\otimes bk_{\sigma}c):=\ma K^{\sigma}(xb)c~\mbox{for all}~x\in E,~b\in\m B,~c\in\m C,~\sigma\in \Omega.$$
Then using Equation \ref{eqn13} we obtain
\begin{align*}
 \langle \nu(x\otimes bk_{\sigma}c),\nu(x'\otimes b'k_{\sigma'}c')\rangle
 &=\langle \ma K^{\sigma}(xb)c,\ma K^{\sigma'}(x'b')c'\rangle=\langle b{k_{\sigma}}c , \langle x, x'\rangle b'{k_{\sigma'}}c' \rangle
 \\&= \langle x\otimes b{k_{\sigma}}c , x'\otimes b'{k_{\sigma'}}c' \rangle
\end{align*}
for all $x,x'\in E;~b,b'\in\m B;~c,c'\in\m C;~\sigma,\sigma'\in S$. Hence $\nu$ is an isometry in this case. This yields a new factorization
for $\mf{K}$-families where $\mf K$ is a CPD-kernel (cf. Section 2 of \cite{Sk12}). In general case (i.e., $\alpha\neq id_{\m B}$) we refer to Theorem \ref{fact}
as a partial factorization theorem for $\mf{K}$-families. 
\end{remark}

\vspace{0.5cm} 
\noindent{\bf Acknowledgements:} Both the authors were supported by Seed Grant from IRCC, IIT Bombay. 
The second author would like to thank 
A. Athavale and F. H. Szafraniec for suggestions. 

\addcontentsline{toc}{chapter}{Bibliography}
\bibliographystyle{amsplain}{
\bibliography{harshbib}}

\end{document}